\documentclass[12pt]{amsart}

\usepackage{amsfonts}
\usepackage{amssymb}

\theoremstyle{plain}
\newtheorem{theorem}{Theorem}[section]

\newtheorem{corollary}[theorem]{Corollary}
\newtheorem{lemma}[theorem]{Lemma}

\theoremstyle{definition}

\newtheorem{question}[theorem]{Question}
\newtheorem{example}[theorem]{Example}

\begin{document}

\baselineskip 5.7mm

\title{Operator norm and numerical radius analogues of Cohen's inequality\footnote{The paper will appear in Mathematical Inequalities and Applications}} 

\author{Roman Drnov\v sek}


\begin{abstract}
Let $D$ be an invertible multiplication operator on $L^2(X, \mu)$, and let 
$A$ be a bounded operator on $L^2(X, \mu)$. In this note we  prove that 
$\|A\|^2 \le \|D A\| \, \|D^{-1} A\|$, where $\|\cdot\|$ denotes the operator norm. 
If, in addition, the operators $A$ and $D$ are positive, we also have 
$w(A)^2 \le w(D A) \, w(D^{-1} A)$, where $w$ denotes the numerical radius.
\end {abstract}

\maketitle

 {\it Math. Subj. Classification (2010)}:  47A30, 47A12,   47A10 \\

\section{Introduction}

Let $A$ be a nonnegative matrix and $D$ a diagonal matrix with positive diagonal entries. 
J. E. Cohen \cite[inequality (3.7)]{Co14a} proved that 
\begin{equation}
r(A)^2 \le r(D A) \, r(D^{-1} A)
\label{Cohen}
\end{equation}
where $r$ denotes the spectral radius. 
In fact, he proved slightly more general inequality \cite[inequality (3.6)]{Co14a}. 
Let $D_1$, $\ldots$, $D_m$ be diagonal matrices with positive diagonal entries such that 
$D_1 \cdots D_m = I$, where $I$ is the identity matrix. Then 
\begin{equation}
 r(A)^m \le r(D_1 A) \, r(D_2 A) \cdots r(D_m A) . 
\label{Cohen_more}
\end{equation}

This inequality is important in the theory of population dynamics 
in Markovian environments; see \cite{Co14b}.  
In this note we consider this inequality with the spectral radius replaced by the operator norm and by the numerical radius. 
In fact, we introduce a more general setting.

Throughout the note, let $\mu$ be a $\sigma$-finite positive measure on a set $X$. 
We consider bounded (linear) operators on the complex Banach space $L^p(X, \mu)$ ($1 \le p \le \infty$). 
The adjoint of an operator $A$ on  $L^p(X, \mu)$ is denoted by $A^*$. 
An operator $A$ on  $L^p(X, \mu)$ is said to be {\it positive} 
if it maps nonnegative functions to nonnegative ones. Given operators $A$ and $B$ on $L^p(X, \mu)$,
we write $A \ge B$ if the operator $A-B$ is positive.
The norm in $L^p(X, \mu)$ and the operator norm are denoted by $\| \cdot \|_p$
and $\| \cdot \|$, respectively. The {\it numerical radius} of an operator $A$ on $L^2(X, \mu)$ is defined by 
$$ w(A) := \sup \{ | \langle A f, f \rangle | : f \in L^2(X, \mu), \| f \|_2 = 1 \} . $$
If, in addition, $A$ is positive, then we have 
$$ w(A) = \sup \{ \langle A f, f \rangle  : f \in L^2(X, \mu) , f \ge 0,  \| f \|_2 = 1 \} . $$
Indeed, this follows from the estimate
$$  | \langle A f, f \rangle | \le \int_X \! |A f| \, |f| \, d\mu \le 
    \langle A |f|, |f| \rangle  $$
that holds for any  $f \in L^2(X, \mu)$. It is well-known \cite{GD97} that 
$$ r(A) \le w(A) \le \|A\| $$
for all bounded operators $A$ on $L^2(X, \mu)$.  
 
We will make use of the following generalized H\"{o}lder's inequality; see e.g. \cite[p. 196, Exercise 31]{Fo99},
or \cite{Wiki} for its proof.

\begin{lemma}
Assume that $r  \in [1,\infty]$ and $p_1$, $\ldots$, $p_m \in [1,\infty]$ satisfy the equality 
$$ \sum_{i=1}^m \frac{1}{p_i} = \frac{1}{r} , $$
where (as usual) we interpret $1/\infty$ as $0$. 
If  $f_i \in L^{p_i}(X, \mu)$ for $i=1, \ldots, m$, then $f_1 \cdots f_m \in L^r(X, \mu)$ and 
$$ \| f_1 \cdots f_m \|_{r} \le \|f_1\|_{p_1} \cdots \|f_m\|_{p_m} \ . $$
\label{Holder}
\end{lemma}

 

\section{Results}

We begin with the operator norm analogue of Cohen's inequality \eqref{Cohen_more}.

\begin{theorem}
\label{main_norm}
Let $D_1$, $\ldots$, $D_m$ be multiplication operators on $L^p(X, \mu)$ ($1 \le p \le \infty$) 
such that $D_1 \cdots D_m = I$.
Let $A$ be a bounded operator $A$ on  $L^p(X, \mu)$. Then
\begin{equation}
\|A\|^m \le \|D_1 A\| \, \|D_2 A\| \cdots \|D_m A\| . 
\label{norm_left}
\end{equation}
If $A$ is the adjoint operator of an operator, then we also have
\begin{equation}
\|A\|^m \le \|A D_1\| \, \|A D_2\| \cdots \|A D_m\| .
\label{norm_right}
\end{equation}

\end{theorem}

\begin{proof}
For each $ i =1, \ldots, m$, let $d_i$ be the function in $L^\infty(X,  \mu)$ such that $D_i f = d_i f$ for 
all $f \in L^p(X, \mu)$. Therefore, $d_1 \cdots d_m = 1$ a.e. on $X$.
There is no loss of generality in assuming that $A  \neq 0$.
Choose an arbitrary number $c \in (0,\|A\|)$. Then there exists a function $f \in L^p(X, \mu)$ such that $\|f\|_p = 1$ 
and the function $g :=A f$ has norm more than $c$. 
Since  $\|D_i A\| \ge \|D_i A f\|_p = \|d_i g\|_p$, we have 
$$ \|D_1 A\| \, \|D_2 A\| \cdots \|D_m A\| \ge \|d_1 g\|_p \, \|d_2 g\|_p \cdots \|d_m g\|_p . $$
Now Lemma \ref{Holder} gives the inequality
$$  \|d_1 g\|_p  \, \|d_2 g\|_p \cdots \|d_m g\|_p \ge 
\| (d_1 g) \cdots (d_m g)\|_{p/m} = \| g^m\|_{p/m} = \|g\|_p^m > c^m \ . $$ 
It follows that 
$$ \|D_1 A\| \, \|D_2 A\| \cdots \|D_m A\| > c^m . $$
Since the number $c \in (0,\|A\|)$ is arbitrary, we obtain the inequality \eqref{norm_left}.

To prove the inequality \eqref{norm_right}, we assume first that $p < \infty$. Then the adjoint operators $A^*$, 
$D_1^*$, $\ldots$, $D_m^*$ are operators on the Banach space $L^q(X, \mu)$, 
where $q \in (1, \infty]$ is the conjugate exponent to $p$. Applying  the inequality \eqref{norm_left} for them, we have
$$ \|A\|^m = \|A^*\|^m \le \|D_1^* A^*\| \cdots \|D_m^* A^*\| = 
\|(A D_1)^*\| \cdots \|(A D_m)^*\| =  \|A D_1\| \cdots \|A D_m\| , $$
proving the inequality \eqref{norm_right} in this case. 

Assume now that $p= \infty$. We are assuming that there exists an operator $B$ on $L^1(X, \mu)$ such that $B^* = A$.
Let $E_i$ ($i=1, \ldots, m$) be the multiplication operator on $L^1(X, \mu)$
with the function $d_i$, so that $E_i^* = D_i$.  Applying  the inequality \eqref{norm_left} for the operators 
$B$, $E_1$, $\ldots$, $E_m$,  we obtain that 
$$ \|A\|^m = \|B\|^m \le \|E_1 B\| \cdots \|E_m B\| =  
\|(E_1 B)^*\| \cdots \|(E_m B)^*\| =  \|A D_1\| \cdots \|A D_m\| . $$
This completes the proof of the theorem.
\end{proof}

\begin{corollary}
Let $D$ be an invertible multiplication operator on $L^p(X, \mu)$ ($1 \le p < \infty$), and let 
$A$ be a bounded operator on  $L^p(X, \mu)$. Then 
\begin{equation}
\|A\|^2 \le \|D A\| \, \|D^{-1} A\| 
\label{norm_left_two}
\end{equation}
and 
$$ \|A\|^2 \le \|A D \| \, \|A D^{-1}\| .  $$
\end{corollary}

We now turn to the numerical radius analogue of Cohen's inequality.

\begin{theorem}
\label{main_numerical}
Let $D_1$, $\ldots$, $D_m$ be positive multiplication operators on $L^2(X, \mu)$ such that $D_1 \cdots D_m \ge I$.
Then, for any positive operator $A$ on  $L^2(X, \mu)$,
\begin{equation}
w(A)^m \le w(D_1 A) \, w(D_2 A) \cdots w(D_m A) 
\label{numerical_left}
\end{equation}
and 
\begin{equation}
w(A)^m \le w(A D_1) \, w(A D_2) \cdots w(A D_m) 
\label{numerical_right}
\end{equation}

\end{theorem}

\begin{proof}
For each $ i =1, \ldots, m$, let $d_i$ be the function in $L^\infty(X,  \mu)$ such that $D_i f = d_i f$ for all 
$f \in L^2(X, \mu)$.
By the assumption, we have $d_1 \cdots d_m \ge 1$ a.e. on $X$.
There is no loss of generality in assuming that $A \neq 0$.
Choose an arbitrary number $c \in (0,w(A))$. 
Then there exists a nonnegative function $f \in L^2(X, \mu)$ such that $\|f\|_2 = 1$ 
and the nonnegative function $g :=A f$ satifies the inequality $ \langle g, f \rangle > c$. 
Since  $w(D_i A) \ge \langle D_i A f, f \rangle  = \langle d_i g, f \rangle = \| \sqrt{d_i g f} \|_2^2$,  we have 
$$ w(D_1 A) \, w(D_2 A) \cdots w(D_m A)  \ge 
\left( \| \sqrt{d_1 g f} \|_2 \, \| \sqrt{d_2 g f} \|_2 \cdots \| \sqrt{d_m g f} \|_2 \right)^2. $$
Using Lemma \ref{Holder} we obtain the inequality
$$   \| \sqrt{d_1 g f} \|_2 \, \| \sqrt{d_2 g f} \|_2 \cdots \| \sqrt{d_m g f} \|_2 
\ge \left( \int_X (d_1 \cdots d_m)^{1/m} g f \, d\mu \right)^{m/2} \ge $$
$$ \ge \left( \int_X g f \, d\mu \right)^{m/2} > c^{m/2}  .  $$ 
It follows that 
$$ w(D_1 A) \, w(D_2 A) \cdots w(D_m A)  > c^m . $$
Since the number $c \in (0,w(A))$ is arbitrary, we get the inequality \eqref{numerical_left}.

To prove \eqref{numerical_right}, we apply \eqref{numerical_left} for the adjoint operator $A^*$:
$$ w(A)^m = w(A^*)^m \le w(D_1 A^*) \cdots w(D_m A^*) =  w(A D_1) \cdots w(A D_m) .  $$
\end{proof}

\begin{corollary}
\label{numerical_two}
Let $D$ be an invertible positive multiplication operator on $L^2(X, \mu)$, and let 
$A$ be a positive operator on  $L^2(X, \mu)$. Then 
\begin{equation}
w(A)^2 \le w(D A) \, w(D^{-1} A) 
\label{numerical_left_two}
\end{equation}
and 
$$ w(A)^2 \le w(A D) \, w(A D^{-1}) .  $$
\end{corollary}

The following example shows that in Theorem \ref{main_numerical} and Corollary \ref{numerical_two}
we cannot omit the assumption that multiplication operators are positive. The same example also shows that 
Cohen's inequality does not hold without the positivity assumption.

\begin{example}
Define the matrices $A$ and $D$ by
$$ A = \left( \begin{matrix}
1 & 1 \cr
1 & 1 
\end{matrix} \right)  
\ \ \ \textrm{and} \ \ \
 D = \left( \begin{matrix}
1 & 0 \cr
0 & -1 
\end{matrix} \right)  .
$$
One can verify that  $r(A) = \|A\| = 2$. 
Since $r(A) \le w(A) \le \|A\|$, we conclude that $w(A) = 2$ as well.
Since 
$$ D A = D^{-1} A = 
\left( \begin{matrix}
1 & 1 \cr
-1 & -1 
\end{matrix} \right)   $$
is unitarily equivalent to a multiple of a Jordan nilpotent $J$ and $w(J) = 1/2$, we have 
$w(D^{-1} A) = w(D A) = \| D A \| /2 = 1$,
and so the inequality \eqref{numerical_left_two} does not hold.
Since $r(D^{-1} A) = r(D A) = 0$, the inequality \eqref{Cohen} is not true either.
\end{example}

One may ask whether the inequality 
$$ w(A)^2 \le w(D A) \, w(A D^{-1}) $$
holds for an invertible positive multiplication operator $D$ on $L^2(X, \mu)$ and for 
a positive operator $A$ on  $L^2(X, \mu)$. The following example show that this is not the case.

\begin{example}
Define the matrices $A$ and $D$ by
$$ A = \left( \begin{matrix}
0 & 1 \cr
0 & 0 
\end{matrix} \right)  
\ \ \ \textrm{and} \ \ \
 D = \left( \begin{matrix}
d & 0 \cr
0 & 1 
\end{matrix} \right) ,
$$
where $d \in (0,1)$. Then $w(A) = w(A D^{-1}) = 1/2$ and 
$w(D A)= d/2$, so that $ w(A)^2 > w(D A) \, w(A D^{-1}) $.
\end{example}

We conclude this note by posing an open question.

\begin{question}
Is Cohen's inequality \eqref{Cohen} true for operators on the space $L^2(X, \mu)$? That is, does it the inequality 
\begin{equation}
r(A)^2 \le r(D A) \, r(D^{-1} A)
\label{Cohen2}
\end{equation}
hold for an invertible positive multiplication operator $D$ on $L^2(X, \mu)$ and for 
a positive operator $A$ on  $L^2(X, \mu)$?
\end{question}

If the operator $A$ has rank one, the answer is affirmative. Namely, if $A = u \otimes v$ for some nonnegative functions 
$u$ and $v$ in $L^2(X, \mu)$ and if the positive function $\varphi \in L^\infty(X, \mu)$ corresponds 
to the multiplication operator $D$,  then we have 
$$ r(A) =  \int_X u v \, d\mu = 
\int_X \sqrt{\varphi u v} \sqrt{u v / \varphi}  \, d\mu , $$
$$ r(D A) =  \int_X \varphi u v \, d\mu \ , \ \ 
r(D^{-1} A) =  \int_X u v / \varphi \, d\mu $$
so that \eqref{Cohen2} holds by the Cauchy-Schwarz inequality.

\vspace{3mm}
{\it Acknowledgment.} The author acknowledges the financial support from the Slovenian Research Agency  (research core funding No. P1-0222).

\vspace{2mm}

\baselineskip 6mm
\noindent
Roman Drnov\v sek \\
Department of Mathematics \\
Faculty of Mathematics and Physics \\
University of Ljubljana \\
Jadranska 19 \\
SI-1000 Ljubljana, Slovenia \\
e-mail : roman.drnovsek@fmf.uni-lj.si 


\begin{thebibliography}{9999}

\bibitem{Co14a}  J. E. Cohen, Cauchy inequalities for the spectral radius of products of diagonal and nonnegative matrices, Proc. Amer. Math. Soc. 142 (2014), no. 11, 3665--3674.

\bibitem{Co14b} J. E. Cohen, Stochastic population dynamics in a Markovian environment implies Taylor's power law of fluctuation scaling, Theoret. Population Biol. 93 (2014), 30--37.

\bibitem{Fo99} G. B. Folland, {\it Real analysis. Modern techniques and their applications}, Second edition, 
John Wiley \& Sons, New York, 1999.

\bibitem{GD97} K.E. Gustafson, D.K.M. Rao, {\it Numerical range. The field of values of linear operators and matrices}, 
  	Universitext, Springer-Verlag, New York, 1997.

\bibitem{Wiki} Generalization of Holder's inequality,  In Wikipedia. Retrieved January 31, 2019, from
https://en.wikipedia.org/wiki/Holder's inequality

\end{thebibliography}
\end{document}